\documentclass[a4paper, 12pt]{article}
\usepackage[a4paper, margin=1in]{geometry}
\usepackage{enumerate}

\usepackage{setspace}
\usepackage{multicol}
\usepackage{multirow}
\usepackage[compact]{titlesec}

\usepackage[numbers]{natbib}
\usepackage{hyperref}
\usepackage{url}
\usepackage{enumitem}
\usepackage{amsmath}
\usepackage{verbatim}
\usepackage{amsthm}
\usepackage{amsfonts}
\usepackage{centernot}
\usepackage{mathrsfs}
\usepackage{mathtools}
\usepackage{amssymb}
\usepackage{graphicx}
\usepackage{upgreek}
\usepackage[justification=centering]{caption}

\newtheorem{theorem}{Theorem}
\theoremstyle{plain}

\newtheorem{conjecture}[theorem]{Conjecture}
\newtheorem{corollary}[theorem]{Corollary}
\newtheorem{lemma}[theorem]{Lemma}
\newtheorem{proposition}[theorem]{Proposition}
\theoremstyle{definition}

\newcommand{\lc}[2]{#1 \leq_{\text{LC}} #2}
\newcommand{\notlc}[2]{#1 \not\leq_{\text{LC}} #2}

\def\lcposet{ \leq_{\text{LC}}}
\def\lcmax{\(\leq_{\text{LC}}\)}
\def\LC{\textnormal{\textbf{LC}}}
\def\poset{\binom{[2k]}{k}^{\lcposet}}

\newcommand{\twopartdefother}[3]
{
	\left\{
		\begin{array}{ll}
			#1 & \mbox{if } #2 \\
			#3 & \mbox{otherwise.}
		\end{array}
	\right.
}
\newcommand{\defeq}{\vcentcolon=}

\renewcommand{\leq}{\leqslant}
\renewcommand{\geq}{\geqslant}

\newcommand{\F}{\mathcal{F}}
\newcommand{\G}{\mathcal{G}}
\newcommand{\B}{\mathcal{B}}
\newcommand{\X}{\mathcal{X}}
\newcommand{\W}{\mathcal{W}}
\newcommand{\Pk}{\mathcal{P}_k}
\newcommand{\M}[1]{{\mathcal{M}}_{#1}}

\DeclarePairedDelimiter\floor{\lfloor}{\rfloor}
\newcommand{\total}[1]{\totalsum #1 }
\DeclareMathOperator{\totalsum}{\Sigma}
\begin{document}
\singlespacing

\title{Two new results on maximal left-compressed intersecting families}
\author{Allan Flower\thanks{School of Mathematics, University of Birmingham, UK. {\tt axf774@student.bham.ac.uk}} \and Richard Mycroft\thanks{School of Mathematics, University of Birmingham, UK. {\tt r.mycroft@bham.ac.uk}. RM is grateful for financial support from EPSRC Standard Grant EP/R034389/1.}}

\maketitle
%
%
\begin{abstract}
    This paper presents two new results on the theory of maximal left-compressed intersecting families (MLCIFs). First, we answer a question raised by Barber by showing that the number of \(k\)-uniform MLCIFs on a ground set of size \(n\) grows as a doubly-exponential function of \(k\), which we identify up to a log factor in the exponent. Among these MLCIFs we identify \(k\) specific MLCIFs --- which we call the canonical MLCIFs --- as being in a meaningful way the most important MLCIFs. Specifically, our second main result shows that the canonical MLCIFs are precisely those which can have maximum weight among all \(k\)-uniform MLCIFs under a non-trivial increasing weight function, and moreover that each canonical MLCIF is the unique \(k\)-uniform MLCIF of maximum weight for some increasing weight function. This gives an interesting generalisation of the Erd\H{o}s--Ko--Rado theorem to a notion of size which places greater significance on some elements of the ground set than others.
\end{abstract}
%
%
\section{Introduction} \label{section:introduction}

The Erd\H{o}s--Ko--Rado theorem \cite{Erdos_Ko_Rado} --- one of the earliest and most influential results in Extremal Set Theory ---  states that if \(\F \subseteq \binom{[n]}{k}\) is intersecting and \(n \geq 2k\), then \(|\F|~\leq~\binom{n-1}{k-1} \). For \(n > 2k\), this upper bound is only attained by stars centred at a single element~\(i\), that is, the families \(S_i \defeq \{ F \in \binom{[n]}{k} \mid i \in F \} \) for each \(i \in [n]\). On the other hand, for \(n = 2k\) there are many intersecting families that attain the Erd\H{o}s--Ko--Rado upper bound, each containing exactly one set from every pair of complementary sets \(F, F^{\text{c}} \in \binom{[2k]}{k}\) where \(F^{\text{c}} \defeq [2k] \setminus F\).

The original proof of the Erd\H{o}s--Ko--Rado theorem used a technique called compression (also known as shifting) which has proved very fruitful in the theory of extremal set systems. Given sets \(G= \{y_i\}_{i \in [k]} \) and \(F = \{x_i\}_{i \in [k]} \) in \(\binom{[n]}{k}\) whose elements are ordered so that \(y_1 < \dots < y_k\) and \(x_1 < \dots < x_k\), say that \(G\) is a \emph{left-compression} of \(F\) (denoted \(\lc{G}{F}\)) if \(y_i \leq x_i\) for each \(i \in [k]\). Similarly, a family \(\F \subseteq \binom{[n]}{k}\) is \emph{left-compressed} if \(\F\) is closed under left-compressions (in other words, for each set \(F \in \F\), every set \(F' \in \binom{[n]}{k}\) with \(\lc{F'}{F}\) is also contained in \(\F\)). Observe that the relation \(\leq_{\text{LC}} \) is a partial order on \(\binom{[n]}{k}\). 

A \(k\)-uniform \emph{maximal left-compressed intersecting family (MLCIF)} on \([n]\) is an intersecting family \(\mathcal{F} \subseteq \binom{[n]}{k}\) which is left-compressed and which is maximal with respect to inclusion. It is a well-known standard fact that for any intersecting family \(\F \subseteq \binom{[n]}{k}\) there exists a left-compressed intersecting family \(\F' \subseteq \binom{[n]}{k}\) with \(|\F'| = |\F|\). (This follows from the fact that a left-compression can be achieved by a sequence of individual shifts, each of which preserves the intersecting property of a family. For more details we recommend the survey by Frankl~\cite{Frankl_shifting_survey} which gives a comprehensive overview of properties and applications of left-compressions.)
It follows that for each \(n\) and \(k\) there is a \(k\)-uniform MLCIF on \([n]\) which is a largest \(k\)-uniform intersecting family on \([n]\), so MLCIFs are a natural object of study when considering the Erd\H{o}s--Ko--Rado theorem and related questions concerning intersecting families. This paper presents two new results on the theory of MLCIFs.

\subsection{The number of MLCIFs on \texorpdfstring{\(n\)}{n} elements.}
Let \(\M{k}\) denote the set of all \(k\)-uniform MLCIFs on \([2k]\). \citet{barber-max_hitting} showed that for \(n \geq 2k\) every MLCIF in \(\M{k}\) extends to a unique \(k\)-uniform MLCIF on \([n]\), and each \(k\)-uniform MLCIF on \([n]\) arises in this way; this result is included as Lemma~\ref{lem:extension lemma} of this paper. In particular for \(n \geq 2k\) the number of distinct \(k\)-uniform MLCIFs on \([n]\) is precisely \(|\M{k}|\), so depends only on \(k\) and not on \(n\). It is then straightforward to give a doubly-exponential upper bound on this quantity. Indeed, \(\binom{[2k]}{k}\) is the union of \(\frac{1}{2}\binom{2k}{k}\) pairs of complementary sets, and any maximal intersecting family \(\F \subseteq \binom{[2k]}{k}\) includes precisely one set from each such pair; it follows that \(|\M{k}| \leq 2^{\frac{1}{2}\binom{2k}{k}}\).

\citet{barber-blog} (see also \cite{barber-thesis}) used Haskell to compute \(|\M{k}|\) for small values of \(k\), and entered these as entry A300099 of the Online Encyclopedia of Integer Sequences~\cite{oeis-barber}. He remarked that this number appears to grow as a double exponential in \(k\), and asked for a proof that this is indeed the case. Our first main result gives an affirmative answer to this question, and gives asymptotic bounds on \(|\M{k}|\) differing only by a log factor in the exponent.

\begin{theorem} \label{thm:main_double-exp}
    For each integer \(k \geq 2\) we have \(2^{\frac{1}{2}\binom{k-1}{\floor{k/2}}} \leq |\M{k}| \leq 2^{\frac{1}{2}\binom{2k}{k}}\). Moreover, for \(k\) sufficiently large we have 
    \(
    2^{\frac{1}{9}k^{-3/2}\binom{2k}{k}} < |\M{k}| < 2^{7\log_2(k)k^{-3/2}\binom{2k}{k}}
    \).
\end{theorem}

We give the proof of Theorem~\ref{thm:main_double-exp} in Section~\ref{section: double exponential}.

\subsection{Optimal MLCIFs under increasing weight functions.}

One interesting line of research extending the Erd\H{o}s--Ko--Rado theorem asks for largest intersecting families on \([n]\) under notions of size which place greater significance on some elements of \([n]\) than on others. \citet{borg-max_hitting} observed that for this it is natural to consider largest MLCIFs rather than arbitrary intersecting families, since otherwise the answer will typically be that the star centred at the most valuable element is trivially the optimal solution. Consequently, Borg considered the following problem. For a fixed set \(X \subseteq [n]\), what are the largest \(k\)-uniform MLCIFs \(\mathcal{F}\) on \([n]\) if our measure of size only counts those elements of \(F\) which intersect \(X\)? If \(1 \in X\) then it follows immediately from the Erd\H{o}s--Ko--Rado that the star centred at 1 is the unique largest \(k\)-uniform MLCIF in this sense. However, sets \(X\) with \(1 \notin X\) give more interesting behaviour, and the problem was resolved (for sufficiently large \(n\)) through a series of works by \citet{borg-max_hitting}, \citet{barber-max_hitting} and \citet{Bowtell&Mycroft-max_hitting}.

The second main result of this paper concerns a similar variation of the Erd\H{o}s--Ko--Rado problem, in which we allow the `value' of an element of \([n]\) to vary continuously in the following way. Fix \(k, n \in \mathbb{N}\) and a weight function  
\(\omega: [n] \rightarrow \mathbb{R}_{\geq 0}\). For each \(F=\{x_i\}_{i \in [k]} \in \binom{[n]}{k}\) and each family \(\F = \{F_1,\dots,F_m\} \subseteq \binom{[n]}{k}\), define
\[
    \omega(F) = \prod_{i=1}^{k} \omega(x_i) \textnormal{ and } \omega(\F) = \sum_{i=1}^{m} \omega(F_i)
\]
to be the weights of \(F\) and \(\F\) respectively. So the weight of a set is the product of the weights of its elements, whilst the weight of a family of sets is the sum of the weights of the sets within the family. 

We say that a weight function \(\omega\) is \emph{trivial} if every \(k\)-uniform MLCIF \(\F\) on \([n]\) has \(\omega(\F) = 0\), and \emph{non-trivial} otherwise (note that this depends on the value of \(k\) but this will always be clear from the context). Also, we say that a \(k\)-uniform MLCIF \(\F\) is \emph{optimal} for~\(\omega\) if \(\omega(\F) \geq \omega(\mathcal{G})\) for every \(k\)-uniform MLCIF \(\mathcal{G}\), and that \(\F\) is \emph{uniquely optimal} for~\(\omega\) if this inequality is strict for every \(k\)-uniform MLCIF \(\mathcal{G}\) with \(\mathcal{G} \neq \F\). Finally, we say that \(\omega\) is \emph{increasing} if \(\omega(i) \leq \omega(j)\) whenever \(i \leq j\). It is natural to consider MLCIFs of maximum weight under increasing weight functions since the goal of maximising the weight of a family is in direct opposition to the requirement that the family is left-compressed: the latter favours elements with lower weights, whilst the increasing weight function places greater significance on higher-weight elements. Note that if \(n < 2k\) then the family \(\F = \binom{[n]}{k}\) is trivially an optimal MLCIF for every weight function, so as in the Erd\H{o}s--Ko--Rado theorem the interesting problem is the case \(n \geq 2k\).

For fixed \(k, n \in \mathbb{N}\) (which will always be clear from the context) and for each \(i \in [k]\), define 
    \[
        \langle i \rangle \defeq \big\{ F \in \binom{[n]}{k} : |F \cap [2i-1]| \geq i \big\},
    \]
and observe that if \(n \geq 2k\) then \(\langle i \rangle\) is a \(k\)-uniform MLCIF on \([n]\). We call \(\langle i \rangle\) the \emph{\(i^{\text{th}}\) canonical MLCIF}. Amongst the great number of \(k\)-uniform MLCIFs on \([n]\) (which is doubly-exponential in \(k\) by Theorem~\ref{thm:main_double-exp}), the \(k\) canonical MLCIFs are in a meaningful sense the principal MLCIFs. Indeed, we shall see in Section~\ref{section:key ideas} that the canonical MLCIFs are precisely those MLCIFs which are formed by left-compressions of a single set. Moreover, our second main result, Theorem~\ref{thm:main_increasing}, demonstrates that for sufficiently large \(n\) only the canonical MLCIFs can be optimal under a non-trivial increasing weight function.

\begin{theorem} \label{thm:main_increasing}
Fix \(k \in \mathbb{N}\) and let \(n \geq 3k^3\binom{2k}{k}\). If \(\omega: [n] \rightarrow \mathbb{R}_{\geq 0}\) is a non-trivial increasing weight function, and \(\F\) is a \(k\)-uniform MLCIF on \([n]\) which is optimal for \(\omega\), then \(\F\) is a canonical MLCIF. Moreover, for each canonical MLCIF \(\F\) there exists an increasing weight function \(\omega\) for which \(\F\) is uniquely optimal for \(\omega\).
\end{theorem}

We prove Theorem~\ref{thm:main_increasing} in Section~\ref{section:weight functions}.

\subsection{Notation}
The following notation is used throughout this paper. For \(m, n \in \mathbb{N}\) we write \([n]\) to denote the set \(\{1, 2, \dots, n\}\) and \([m, n]\) to denote the set \mbox{\(\{m, m+1, \dots , n\}\)}. For a set \(X\), the set of all subsets of \(X\) of size \(k\) is denoted by \(\binom{X}{k}\). We use upper-case letters in italic font (such as \(F\) and \(G\)) for finite sets of natural numbers, and calligraphic characters (such as \(\F\) and \(\G\)) to denote families of sets. When writing a set \(F = \{x_1, \dots, x_k\}\), it is always to be assumed that \(x_i < x_j\) whenever \(i < j\), and this set is often written as \(\{x_i\}_{i \in [k]}\). When \(F \in \binom{[2k]}{k}\), we denote the complement of \(F\) by \(F^{\text{c}} \defeq [2k] \setminus F\).

\section{Key ideas}\label{section:key ideas}
We begin by presenting two key results on MLCIFs which were given by \citet{barber-max_hitting}. The first is that for \(n \geq 2k\) there is a one-to-one correspondence between MLCIFs on \([n]\) and MLCIFs on \([2k]\).

\begin{lemma}[{\cite[Lemma~8]{barber-max_hitting}}] \label{lem:extension lemma} 
    Let \(n \geq 2k\). For each MLCIF \(\F \subseteq \binom{[2k]}{k}\) there is a unique MLCIF \(\G(\F) \subseteq \binom{[n]}{k}\) with \(\F \subseteq \G(\F)\). Moreover, every MLCIF \(\G \subseteq \binom{[n]}{k}\) has \(\G = \G(\F)\) for a unique MLCIF \(\F \subseteq \binom{[2k]}{k}\). 
\end{lemma}

As noted in the introduction, it follows from Lemma~\ref{lem:extension lemma} that the number of \(k\)-uniform MLCIFs on \([n]\) is precisely \(|\M{k}|\), the number of \(k\)-uniform MLCIFs on \([2k]\). The second result of Barber that we use is the following fact which was established in the proof of Lemma~\ref{lem:extension lemma} as a consequence of the maximality of MLCIFs.

\begin{lemma}[\cite{barber-max_hitting}] \label{lem:n_expansion_lemma}
    Let \(\F \subseteq \binom{[n]}{k}\) be an MLCIF. For each \(j \in [k]\), if \(F = \{x_i\}_{i \in [k]} \in \F\) has \(k-j\) many elements strictly greater than \(k+j\), then \(\F\) contains the element \(F' = \{x_i\}_{i \in [j]} \cup [n-k+j+1, n]\).
\end{lemma}

In other words, if \(F= \{x_i\}_{i \in [k]} \in \F\) and \(x_{j+1} > k+j\), then \(\F\) contains the set \(F'\) whose smallest \(j\) elements are the smallest \(j\) elements of \(F\) and whose remaining \(k-j\) elements are the largest \(k-j\) elements of \([n]\). For example, if a 5-uniform MLCIF \(\F\) on \(n \geq 10\) elements contains the set \( \{x,y,z,9,10\} \), then it also contains the set \( \{x,y,z, n-1,n\} \).

\subsection{Boundary sets}
Let \(\F\) be a \(k\)-uniform MLCIF on \([n]\). We refer to the \lcmax-maximal elements of \(\F\) as the \emph{boundary sets} of \(\F\). So \(B \in \F\) is a boundary set if there there is no \(F \in \F\) such that \(F \neq B\) and \(\lc{B}{F}\). Let \(\B(\F)\) be the collection of boundary sets of \(\F\). It follows that
\begin{equation} \label{eqn:boundary}
     \F = \bigcup_{B \in \B(\F)} \LC(B),
\end{equation}
where \(\LC(B)\) denotes the family of all sets \(F \in \binom{[n]}{k}\) with \(\lc{F}{B}\). 
From \eqref{eqn:boundary}, it follows that for distinct \(\F,\G \in \binom{[n]}{k}\) we have \(\B(\F) \neq \B(\G)\).
One crucial property of canonical MLCIFs is that they are precisely the MLCIFs with a unique boundary set. Indeed, for each \(i \in [k]\) define \[Z_i \defeq [i,2i-1]\cup[n-k+i+1,n]\] (we suppress the dependence on \(k\) and \(n\) as this will always be clear from the context). We then have~\(\langle i \rangle = \LC(Z_i)\), so \(Z_i\) is the unique boundary set of the canonical MLCIF~\(\langle i \rangle\). The fact that every other MLCIF has at least two boundary sets is asserted by the next proposition.

\begin{proposition} \label{prop:boundary}
Let \(\F \subseteq \binom{[n]}{k}\) be an MLCIF, where \(n \geq 2k\).
\begin{enumerate}[label = (\alph*)]
    \item For every \(F \in \F\) we have \(\lc{F}{Z_i}\) for some \(i \in [k]\).
    \item Either \(\F = \langle i \rangle\) for some \(i \in [k]\) or \(\F\) has at least two boundary sets.
\end{enumerate}
\end{proposition}

\begin{proof}
For (a), suppose that \(\notlc{F}{Z_i}\) for every \(i \in [k]\). This means that we have \(|F \cap [2i-1]| < i\) for each \(i \in [k]\), and it follows that \(\lc{\{2j:j \in [k]\}}{F}\). Since we also have \(\lc{\{2j-1:j \in [k]\}}{\{2j:j \in [k]\}}\), the fact that \(\F\) is left-compressed implies that \(\{2j:j \in [k]\}, \{2j-1:j \in [k]\} \in \F\), contradicting the intersecting property of \(\F\). 

For (b), suppose that \(\F\) has only one boundary set \(G\). This means that every \(F \in \F\) has \(\lc{F}{G}\). By (a) we may fix \(i \in [k]\) with \(\lc{G}{Z_i}\), and we then have \(\lc{F}{Z_i}\) for every \(F \in \F\), so \(\F \subseteq \langle i \rangle\). By maximality of \(\F\) it follows that \(\F = \langle i \rangle\).
\end{proof}

One important corollary of Proposition~\ref{prop:boundary} is that under any non-trivial weight function there is some canonical family with non-zero weight.

\begin{corollary} \label{non-triv}
    Let \(k \in \mathbb{N}\) and \(n \geq 2k\). If \(\omega : [n] \to \mathbb{R}_{\geq 0}\) is a non-trivial weight function, then there is some \(i \in [k]\) for which the canonical family \(\langle i \rangle\) has \(\omega(\langle i \rangle) > 0\).
\end{corollary}

\begin{proof}
Since \(\omega\) is non-trivial there exists an MLCIF \(\F\) on \([n]\) with \(\omega(\F) > 0\). This implies that some set \(F \in \F\) has \(\omega(F) > 0\). By Proposition~\ref{prop:boundary} there is \(i \in [k]\) for which \(\lc{F}{Z_i}\); it follows that \(F \in \langle i \rangle\), and so \(\omega(\langle i \rangle) > 0\).
\end{proof}

\subsection{The type of an MLCIF}
Let \(\F\) be a \(k\)-uniform MLCIF on \([n]\), where \(n \geq 2k\). For each set \( F \in \F\) we define the index of~\(F\), denoted \( \iota(F) \), to be the smallest~\( i \) such that \( \lc{F}{Z_i} \); note that such an \(i\) must exist by Proposition~\ref{prop:boundary}(a). We say that \(F \in \binom{[n]}{k}\) is a \emph{strong} set in \(\F\) if \(|F \cap [2k+1,n]| = k-\iota(F)\). For example, for \(n > 8\) the \(4\)-uniform MLCIF \(\LC(\{1,2,n-1,n\})\cup\LC(\{2,4,5,n\})\) has \(\{2,4,5,n\}\) as a strong set of index 3. On the other hand, it is straightforward to check that \(\LC(\{1,4,n-1,n\})\cup\LC(\{2,3,4,n\})\) is a \(4\)-uniform MLCIF with no strong set. The next lemma shows that if an MLCIF \(\F\) contains a strong set, then this set has greatest index among all sets in \(\F\).

\begin{lemma} \label{lem:small_generators_are_small}
    Let \(\F\) be a \(k\)-uniform MLCIF on \([n]\), where \(n \geq 2k\). If \(F\) is a strong set in~\(\F\), then every \(G \in \F\) has \(\iota(G) \leq \iota(F)\).
\end{lemma}

\begin{proof}
     Suppose for a contradiction that \(F= \{x_i\}_{i \in [k]}\) is a strong set in \(\F\) and that \(G = \{y_i\}_{i \in [k]} \in \F\) has \(\iota(F) = i < j = \iota(G)\). 
By definition of \(\iota(G)\) we then have \(y_\ell \geq 2\ell\) for every \(\ell < j\), and so 
\begin{equation*}
        G' \defeq \lc{\{2,4,\dots,2(i-1),2i,2i+1,\dots,i+k\}}{G}.
\end{equation*}
Similarly, by definition of \(\iota(F)\) we have \(x_\ell \geq 2\ell\) for every \(\ell < i\), and it follows that \(x_i \geq 2i-1\). We also have \(F' \defeq \{x_1, \dots, x_{i}\} \cup [n-k+i+1, n] \in \F\) by Lemma~\ref{lem:n_expansion_lemma} since~\(F\) is a strong set, and so 
\begin{equation*}
        F'' \defeq \lc{\{1,3,\dots,2i-3,2i-1\} \cup [n-k+i+1, n]}{F'}.
\end{equation*}
Since \(G, F' \in \F\) and \(\F\) is left-compressed it follows that \(G'\) and \(F''\) are both in \(\F\), which contradicts the fact that \(\F\) is intersecting.
\end{proof}

Lemma~\ref{lem:small_generators_are_small} immediately yields the following corollary.

\begin{corollary} \label{cor:strong_generators_all_same}
    Let \(n \geq 2k\). If \(F\), \(F'\) are strong sets in a \(k\)-uniform MLCIF on \([n]\), then 
    \(\iota(F) = \iota(F')\).
\end{corollary}

For each MLCIF \(\F\) which contains at least one strong set, we define the \emph{type} \(\iota(\F)\) of \(\F\) to be the common index \(\iota(F)\) of each strong set in \(\F\). This signifies that \(\F\) is ``close'' to the canonical MLCIF \(\langle \iota(\F) \rangle\) in a certain sense which will be useful in the proof of Theorem~\ref{thm:main_increasing}. If \(\F\) does not contain a strong set then we say that \(\F\) is \emph{typeless}.

\section{MLCIFs under increasing weight functions} \label{section:weight functions}

We begin our proof of Theorem~\ref{thm:main_increasing} with the following theorem.

\begin{theorem} \label{thm:second_increasing}
    Let \(k \in \mathbb{N}\), set \(c_k \defeq 1/\binom{2k}{k}\) and fix \(n \geq 3k^3\binom{2k}{k}\). Also let \(\F\) be a \(k\)-uniform MLCIF on \([n]\) which is not canonical, and \(\omega : [n] \to \mathbb{R}_{\geq 0}\) be a non-trivial increasing weight function. If \(\F\) is typeless, then
    \[
        \omega(\F) < \sum_{j \in [k]} \frac{ \omega(\langle j \rangle)}{k}.
    \]
    Otherwise,
    \[
        \omega(\F) < (1-c_k) \cdot \omega(\langle \iota(\F) \rangle) + c_k\sum_{j \in [k]} \frac{ \omega(\langle j \rangle)}{k}.
    \]
\end{theorem}

\begin{proof}    
    Let \(\mathcal{S} \defeq \{F \in \F : F \textnormal{ is strong}\}\) and for each \(j \in [k]\) let \(\mathcal{W}_j \defeq \{F \in \F \setminus \mathcal{S} : \iota(F) = j\}\). So \(\mathcal{S}\) consists of all strong sets in \(\F\), and \(\mathcal{W}_j\) consists of all sets in \(\F\) of index~\(j\) which are not strong.    
    We then have
    \(\F = \mathcal{S} \cup \bigcup_{j \in [k]} \mathcal{W}_j\), and so 
    \begin{equation}\label{eq0}
        \omega(\F) \leq \omega(\mathcal{S}) + \sum_{j \in [k]} \omega(\mathcal{W}_j).
    \end{equation}
    
    If \(\F\) is typeless (has no strong sets) then \(\mathcal{S} = \emptyset\), and in this case we set \(\iota \defeq k\). Otherwise write \(\iota \defeq \iota(\F)\) to denote the type of \(\F\). Also,   
        for every set \(J \subseteq [n]\) with \(|J| \leq k\) define \(\X_J \defeq \{J \cup G : G \in \binom{[2k+1,n]}{k-|J|}\}\).
    For each \(F \in \mathcal{S}\) we have \(\iota(F) = \iota\) by definition of \(\iota(\F)\), meaning that \(|F \cap [2\iota-1]| \geq \iota\). Together with the fact that \(F\) is strong this implies that \(F \in \X_J\) for some \(J \in \binom{[2\iota - 1]}{\iota}\). Moreover, if \(F \in \X_J\) for \(J = [\iota, 2\iota -1]\), then by Lemma~\ref{lem:n_expansion_lemma} we have \(Z_\iota \in \F\) (recall that \(Z_\iota \defeq [\iota,2\iota -1] \cup [n-k+\iota + 1,n]\) is the unique boundary set of \(\langle \iota \rangle\)). It follows that \(\langle \iota \rangle \subseteq \F\) and therefore \(\F = \langle \iota \rangle\), contradicting the fact that \(\F\) is not a canonical MLCIF. We conclude that  
    \begin{equation} \label{eqS}
    \mathcal{S} \subseteq \left\{ \X_J : F \in \binom{[2\iota - 1]}{\iota} \setminus [\iota, 2\iota - 1] \right\}.
    \end{equation}
    Since each element of \(\X_{[\iota, 2\iota -1]}\) is in \(\langle \iota \rangle\), it follows that 
    \begin{equation} \label{eqX}
        \X_{[\iota, 2\iota -1]} \subseteq \langle \iota \rangle \setminus \mathcal{S}.
    \end{equation}
    We then have
    \begin{align*} 
    \omega(\mathcal{S}) 
        &\leq \sum_{\substack{J \in \binom{[2\iota -1]}{\iota}\\ J  \neq[\iota,2\iota-1]}}\omega(\X_J) 
        \leq \left(\binom{2\iota-1}{\iota} -1 \right) \cdot \omega (\X_{[\iota,2\iota-1]}) 
        \\& \leq \left(\binom{2k}{k}-1\right) \cdot \Big(\omega(\langle\iota\rangle) - \omega(\mathcal{S})\Big)  = \left(\frac{1}{c_k} - 1 \right) \Big(\omega(\langle\iota\rangle) - \omega(\mathcal{S})\Big),
    \end{align*}
    where the first inequality follows from~\eqref{eqS}, the second inequality holds since \(\omega\) is increasing and the third inequality follows from~\eqref{eqX} and the fact that \(\mathcal{S} \subseteq \langle \iota \rangle\).
    By rearranging, we obtain 
    \begin{equation} \label{eq1} \omega(\mathcal{S}) \leq (1-c_k) \cdot \omega(\langle \iota \rangle).
    \end{equation}

    Now consider a set \(F \in \W_j \) for some \(j \in [k]\).
    By definition of \(\W_j\) we know that \(F\) is not strong, 
    so \(|F \cap [2k]| \geq j +1\), and also that \(\iota(F) = j\), so \(F \in \langle j \rangle\). The latter implies that \(|F \cap [2j-1]| \geq j\), so \(\W_j \subseteq \LC(T_j\cup\{2k\})\), where \(T_j = [j,2j-1] \cup[n-k+j+2,n]\).
    It follows that
    \begin{align}
        \omega(\W_j) &\leq \nonumber
        \omega(\LC(T_j \cup \{2k\})) \leq \sum_{\substack{\lc{A}{T_j} \\ x \in [2k] \setminus A}} \omega(A \cup \{x\}) \\& \leq \frac{2k}{n-3k} \sum_{\substack{\lc{A}{T_j} \\ x \in [2k+1,n]\setminus A}} \omega(A \cup \{x\}) \leq \frac{2k^2}{n-3k} \omega(\langle j \rangle ).\label{eq2}
    \end{align}
    Indeed, the second inequality holds because each set in \(\LC(T_j \cup \{2k\})\) can be written as \(A \cup \{x\}\) with \(\lc{A}{T_j}\) and \(x \in [2k] \setminus A\). To see that the third inequality holds, observe that for each set \(\lc{A}{T_j}\) we have \(\sum_{x \in [2k] \setminus A} \omega(A \cup \{x\}) \leq \frac{2k}{n-3k} \sum_{x \in [2k+1,n]\setminus A} \omega(A \cup \{x\})\), because each of the at most \(2k\) terms in the left hand sum is less than or equal to each of the at least \(n-3k\) terms in the right hand sum. Finally, the concluding inequality holds because for each \(\lc{A}{T_j}\) and \(x \in [2k+1,n]\setminus A\) we have \(A \cup \{x\} \in \langle j \rangle\), and moreover each set in \(\langle j \rangle\) can be written in this form in at most \(k\) distinct ways.
    
    Since \(c_k = 1/\binom{2k}{k}\) and \(n \geq 3k^3\binom{2k}{k}\) we have \(n - 3k > 2k^3/c_k\) (note for this that we may assume \(k \geq 2\) since for \(k=1\) the canonical family \(\langle 1 \rangle\) is the only MLCIF so the theorem holds vacuously). So \(\frac{c_k}{k} > \frac{2k^2}{n-3k}\), and it follows by~\eqref{eq2} that in all cases we have \(\omega(\W_j) \leq \frac{c_k}{k} \omega(\langle j \rangle)\), and if \(\omega(\langle j \rangle) > 0\) then \(\omega(\W_j) < \frac{c_k}{k}\omega(\langle j \rangle)\). Since \(\omega\) is non-trivial Corollary~\ref{non-triv} implies that \(\omega(\langle j \rangle) > 0\) for some \(j \in [k]\), and so \(\sum_{j \in [k]} \omega(\W_j) < \frac{c_k}{k} \sum_{j \in [k]} \omega(\langle j \rangle)\). Together with~\eqref{eq0} and~\eqref{eq1} this gives   
    \[
        \omega(\F) \leq \omega(\mathcal{S}) + \sum_{j \in [k]} \omega(\W_j) < (1-c_k) \cdot \omega(\langle \iota\rangle) + \frac{c_k}{k} \sum_{j \in [k]} \omega(\langle j \rangle),
    \]
    giving the second inequality of the theorem statement. For the first inequality, recall that if \(\F\) is typeless then \(\mathcal{S} = \emptyset\) so \(\omega(\mathcal{S}) = 0\). Since \(c_k \leq 1\) it follows that 
       \[
        \omega(\F) \leq \sum_{j \in [k]} \omega(\W_j) < \frac{c_k}{k} \sum_{j \in [k]} \omega(\langle j \rangle) \leq \frac{1}{k} \sum_{j \in [k]} \omega(\langle j \rangle). \qedhere
    \]     
\end{proof}

Observe that in each case the upper bound on \(\omega(\F)\) given by Theorem~\ref{thm:second_increasing} is a biased average of the weights \(\omega(\langle i \rangle)\) for \(i \in [k]\). This immediately gives the following corollary, that at least one canonical MLCIF has greater weight than \(\F\).

\begin{corollary}\label{cor:not_unique}
    Let \(n, k \in \mathbb{N}\) have \(n \geq 3k^3\binom{2k}{k}\), and let \(\omega : [n] \to \mathbb{R}_{\geq 0}\) be a non-trivial increasing weight function. If \(\F\) is a \(k\)-uniform MLCIF on \([n]\) which is not canonical, then 
    \(
        \omega(\F) < \omega(\langle i \rangle)
    \)
    for some \(i \in [k] \).
\end{corollary}

Corollary~\ref{cor:not_unique} gives the first part of Theorem~\ref{thm:main_increasing}, that only canonical MLCIFs can be optimal for an increasing weight function. The following lemma gives the second statement: for every canonical MLCIF \(\F\) there exists an increasing weight function for which \(\F\) is uniquely optimal.

\begin{lemma} \label{lem:canonical_families_unique}
    Let \(n, k \in \mathbb{N}\) have \(n \geq 3k^3\binom{2k}{k}\). For each \(i \in [k]\) the canonical MLCIF \(\langle i \rangle\) is uniquely optimal for the increasing weight function \(\omega_{i}: [n] \rightarrow \mathbb{R}_{\geq 0} \) given by 
    \[
        \omega_{i}(j) = \twopartdefother{0}{j < i,}{1}
    \]
\end{lemma}

\begin{proof}
    Write \(\omega \defeq \omega_{i}\). Since \(\omega\) is increasing, Corollary~\ref{cor:not_unique} tells us that no non-canonical MLCIF is optimal for \(\omega\), so it suffices to show that \(\omega(\langle i \rangle) > \omega(\langle j \rangle) \) for all \(j \neq i\). Observe first that a set \(F \in \binom{[n]}{k}\) has \(\omega(F) = 1\) if \(|F \cap [i-1]| = \emptyset\) and \(\omega(F) = 0\) otherwise, so \(\omega(\F)\) counts the number of sets in the family \(\F\) that do not intersect \([i-1]\). It follows that for each \(j < i\) we have \(\omega(\langle j \rangle) = 0\). On the other hand, for each \(j > i\) we have 
    \begin{align*}
    \omega(\langle j \rangle) 
        & \leq  |\langle j \rangle | \leq  \sum_{\ell = i+1}^{k} \binom{2k}{\ell} \binom{n-2k}{k-\ell} \leq k \cdot \binom{2k}{k}\cdot \binom{n-2k}{k-(i+1)} \\
        & = k \binom{2k}{k} \frac{k-i}{n-3k+i+1}\binom{n-2k}{k-i} \leq \frac{k^2}{n/2} \binom{2k}{k} \omega (\langle i \rangle) < \omega (\langle i \rangle),
    \end{align*}
    as required. Here the second inequality follows from the fact that each set in \(\langle j \rangle\) contains at least \(j \geq i+1\) elements from \([2k]\) and the third inequality from the fact that \(\binom{n-2k}{k-\ell}\) decreases as \(\ell\) increases. The penultimate inequality uses the fact that for every \(G \in \binom{[2k+1, n]}{k-i}\) the set \(F \defeq [i, 2i-1] \cup G\) is an element of \(\langle i \rangle\) with \(\omega(F) = 1\), so \(\omega(\langle i \rangle) \geq \binom{n-2k}{k-i}\), and both this and the final inequality use our lower bound on \(n\).
\end{proof}

Combining Corollary~\ref{cor:not_unique} and Lemma~\ref{lem:canonical_families_unique} establishes Theorem~\ref{thm:main_increasing}.

\section{The number of MLCIFs} \label{section: double exponential}

Recall that \(\M{k}\) denotes the collection of all \(k\)-uniform MLCIFs on \([2k]\). Our goal in this section is to prove Theorem~\ref{thm:main_double-exp}, showing that \(|\mathcal{M}_k|\) grows as a double exponential in \(k\). 
We begin by showing that, 
as one might expect, \(|\M{k}|\) increases with \(k\). Both here and later we use the key observation that a family \(\F \subseteq \binom{[2k]}{k}\) is intersecting if and only if there is no \(F \in \F\) with \(F^c \in \F\). Consequently, if \(\F \subseteq \binom{[2k]}{k}\) is an MLCIF then for each \(F \in \binom{[2k]}{k}\) with \(F \notin \F\) we have \(F^c \in \F\), as otherwise we could extend \(\F\) to a larger left-compressed intersecting family by adding a set \(F \notin \F\) which is \(\leq_{\mathrm{LC}}\)-minimal among all sets \(F \notin \F\) with \(F^c \notin \F\).

\begin{lemma} \label{lem:noofIFsGrow}
   For all \(k \in \mathbb{N}\) we have \(|\M{k}| \leq |\M{k+1}|\).
\end{lemma} 

\begin{proof}
    Consider a family \(\F \in \M{k}\). 
    Let \(\F' \defeq \{ F \cup \{2k+2\} : F \in \F \}\), and observe that the family \(\widehat{\F} = \bigcup_{F \in \F'} \LC(F)\) is a left-compressed and intersecting \((k+1)\)-uniform family on \([2k+2]\). Indeed, for any two sets \(\{x_i\}_{i \in [k+1]}, \{y\}_{i \in [k+1]} \in \widehat{\F}\) we have \( \{x_i\}_{i \in [k+1]} \cap \{y\}_{i \in [k+1]} \supseteq \{x_i\}_{i \in [k]} \cap \{y\}_{i \in [k]} \neq \emptyset\) since \(\{x_i\}_{i \in [k]}, \{y\}_{i \in [k]} \in \F\). It follows that \(\widehat{\F}\) extends to an MLCIF, that is, that \( \widehat{\F} \subseteq \overline{\F}\) for some \(\overline{\F} \in \M{k+1}\). 

    Now consider any \(\F, \G \in \M{k}\) with \(\F \neq \G\). There must be a set \(F \in \F \setminus \G\), and so \(F^{\text{c}} \in \G\) by our observation above. We then have \(F \cup \{2k+2\} \in \overline{\F}\). However, we cannot have \(F \cup \{2k+2\} \in \overline{\G}\) as \(\overline{\G}\) contains the set \(F^{\text{c}} \cup \{2k+1\}\) (since this set is a left-compression of \(F^{\text{c}} \cup \{2k+2\}\)).
    So \(\overline{\F} \neq \overline{\G}\). We conclude that the MLCIFs \(\overline{\F} \in \M{k+1}\) for \(\F \in \M{k}\) are all distinct, and so \(|\M{k}| \leq |\M{k+1}|\).
\end{proof}

For each set \(F \in \binom{[n]}{k}\) we write \(\total{F} \defeq \sum_{x \in F} x\) for the sum of the elements of~\(F\). Observe that if \(\lc{F}{G}\) then \(\total{F} \leq \total{G}\).
For each even integer \(k\) let \(\Pk\) be the collection of all pairs of complementary elements of \(\binom{[2k]}{k}\) that each sum to \(\frac{k}{2}(2k+1)\). In other words \( \Pk \defeq \big\{ \{F, F^c\} : F \in \binom{[2k]}{k} \textnormal{ and }\total{F} = \total{F^c} = \frac{k}{2}(2k+1) \} \big\} \). 

Our interest in \(\Pk\) is due to the next lemma, which gives a lower bound on the size of \(\mathcal{M}_k\) in terms of the size of \(\Pk\).

\begin{lemma} \label{lem:lower_bound}
For each even \(k \in \mathbb{N}\) we have \(|\mathcal{M}_k| \geq 2^{|\Pk|}\).
\end{lemma}

\begin{proof}
Let \(\G \subseteq \binom{[2k]}{k}\) be a family of sets which consists of precisely one set from each pair in \(\Pk\). Observe that the family \(\G^\leq = \bigcup_{G \in \G} \LC(G)\) is both left-compressed and intersecting. 
    To see the latter, observe that if 
    \(\lc{F}{G}\) and \(F \neq G\) then \(\total{F} < \total{G}\). So for each \(F \in \G^\leq \setminus \G\) we have \(\total{F} < \frac{k}{2}(2k-1)\) and \(\total{F^c} = k(2k+1) - \total{F} > \frac{k}{2}(2k-1)\), and thus \(F^c \notin \G^\leq\). Since \(\G\) contains precisely one set from each complementary pair in~\(\Pk\), we also have \(G^c \notin \G^\leq\) for each \(G \in \G\). So \(\G^\leq\) does not contain both \(F\) and \(F^c\) for any \(F \in \binom{[2k]}{k}\), and so \(\G^\leq\) is intersecting as claimed. It follows that there is some MLCIF \(\F(\G) \in \mathcal{M}_k\) with \(\G^\leq \subseteq \F(\G)\). 

There are \(2^{|\Pk|}\) different possible ways to form a family \(\G\) as described above, since we need to choose one set from each pair in \(\Pk\). Moreover, if \(\G \neq \G'\) then there is some pair \(\{F, F^c\} \in \Pk\) with \(F \in \G\) and \(F^c \in \G'\). We then have \(F \in \F(\G)\) and \(F^c \in \F(\G')\), so \(\F(\G) \neq \F(\G')\). So each choice of \(\G\) results in a different MLCIF \(\F(\G) \in \mathcal{M}_k\), and so \(|\mathcal{M}_k| \geq 2^{|\Pk|}\).
\end{proof}

We can now prove the part of Theorem~\ref{thm:main_double-exp} which applies for all \(k \geq 2\). We state this as the following theorem. 
\begin{theorem}\label{thm:main_allk}
    For each integer \(k \geq 2\) we have \(2^{\frac{1}{2}\binom{k-1}{\floor{k/2}}} \leq |\M{k}| \leq 2^{\frac{1}{2}\binom{2k}{k}}\). 
\end{theorem}

\begin{proof} As noted in the introduction, \(\binom{[2k]}{k}\) may be partitioned into \(\frac{1}{2}\binom{2k}{k}\) complementary pairs, and any \(k\)-uniform MLCIF \(\F \subseteq \binom{[2k]}{k}\) must contain precisely one set from each pair. It follows that \(|\mathcal{M}_k| \leq 2^{\frac{1}{2}\binom{2k}{k}}\), and so it remains to demonstrate the lower bound on \(|\mathcal{M}_k|\).
By Lemma~\ref{lem:noofIFsGrow} it is sufficient to show that if \(k\) is even then \(|\mathcal{M}_k| \geq 2^{\frac{1}{2}\binom{k}{k/2}}\). To do this, partition the set \([2k]\) into \(k\) pairs which each sum to \(2k+1\), namely \(\{j, 2k+1-j\}\) for each \(j \in [k]\). Since \(k\) is even we can then form a set \(F \in \binom{[2k]}{k}\) with \(\total{F} = \frac{k}{2}(2k+1)\) by choosing any \(\frac{k}{2}\) of these pairs and taking \(F\) to be their union. We then have \(\{F, F^c\} \in \Pk\). Since there are \(\binom{k}{k/2}\) options for which pairs we choose to form \(F\), we obtain \(|\Pk| \geq \frac{1}{2}\binom{k}{k/2}\). Lemma~\ref{lem:lower_bound} then gives the desired bound. \end{proof}

We now turn to the second part of Theorem~\ref{thm:main_double-exp}, which gives asymptotic bounds on~\(|\M{k}|\). The upper bound given in Theorem~\ref{thm:main_allk} was obtained by counting all the possible ways to choose a set from each of \(\frac{1}{2}\binom{2k}{k}\) complementary pairs. We can obtain a stronger bound on \(|\M{k}|\) by exploiting the fact that the family \(\B(\F)\) of boundary elements of \(\F\) is unique for each \(\F \in \M{k}\). Let \(\poset\) denote the partially ordered set (poset) \((\binom{[2k]}{k}), \lcposet )\). Since the boundary sets of \(\F\) are the \lcmax-maximal sets in \(\F\), the family \(\B(\F)\) is an antichain in \(\poset\). Together with the aforementioned uniqueness property, it follows that if the largest antichain in \(\poset\) has size \(q\), then \(|\M{k}| \leq \sum_{i=1}^{q}\binom{\binom{2k}{k}}{i}\).

We bound the size of the largest antichain in \(\poset\) in Corollary~\ref{cor:boundary}, using the following structural properties. Given a finite poset \(P = (X,\leq)\) and \(x \in X\), define the \emph{rank of \(x\)}, denoted \(\rho(x)\), to be the largest \(j\) such that \(P\) has a chain of length \(j\) with \(x\) as the largest element (with respect to \(\leq\)). Define \emph{rank layer \(j\)} to be \(\rho_j \defeq |\{x \in X : \rho(x) = j\} \), and let \(p_j \defeq |\rho_j|\). Let \(t\) be the largest integer such that \(p_t \neq 0\) (so \(t\) is the length of the longest chain in \(P\)). We say that \(P\) is \emph{rank-symmetric} if \(p_i = p_{t-i}\) for all \(0 \leq i \leq t\), \emph{rank-unimodal} if \(p_0 \leq p_1 \leq \cdots \leq p_m \geq p_{m+1} \geq \cdots \geq p_t\) for some \(0 \leq m \leq t\), and \emph{Sperner} if no antichain in \(P\) is larger than \(\max\{p_i : 0 \leq i \leq t\}\).

We now introduce the finite Young lattice \(L(m,n)\), a poset which has been widely-studied. Our presentation of \(L(m,n)\) is somewhat non-standard, as we give partitions in increasing order and include zero elements to indicate empty parts. The purpose of these changes is to highlight more clearly the correspondence presented in Lemma~\ref{lem:poset-iso} between \(L(m,n)\) and \(\poset\). Given positive integers \(m,n,r\), we say that \((x_i)_{i \in [m]} \defeq (x_1, \dots, x_m) \) is a \emph{partition (of \(r\)) into \(m\) parts with largest part \(n\)} if \(0 \leq x_1 \leq \cdots \leq x_m \leq n\) and \(\sum_{i=1}^{m} x_i = r\). For example, \(01124\) is a partition of 8 into 5 parts with largest part 4 (for brevity, we often write the partition \((x_1, \dots,x_m)\) as \(x_1\cdots x_m\)). Say that \((x_i)_{i \in [m]} \preceq (y_i)_{i \in [m]}\) if \(x_i \leq y_i\) for all \(i \in [m]\). Let \(L(m,n)\) denote the set of all partitions with exactly \(m\) parts, each part having size at most \(n\). For example, 
\(L(2,3) = \{ 00,01,02,03,11,12,13,22,23,33 \}\). 
Then \(( L(m,n), \preceq)\) (abbreviated to just \(L(m,n)\)), is a poset with rank function \(\rho\left((x_i)_{i \in [m]}\right) = \sum_{i=1}^{m} x_i\). The following lemma gives an equivalency between\(\poset\) and \(L(k,k)\).

\begin{lemma}\label{lem:poset-iso}
    For all \(k \in \mathbb{N}\), \( \poset\) and \(L(k,k)\) are order isomorphic.
\end{lemma}

\begin{proof}
    It suffices to show that the map \(\phi : \poset \rightarrow L(k,k)\) given by \(\phi(\{x_i\}_{i \in [k]}) \defeq (x_i - i)_{i \in [k]}\) is an order isomorphism. Consider \(X,Y\in \binom{[2k]}{k}\) with \(X = \{x_i\}_{i \in [k]} \) and \(Y = \{y_i\}_{i \in [k]} \). To see that \(\phi\) is well-defined, recall that our notation convention requires that \(0 < x_1 < x_2 < \dots < x_k \leq 2k\), so for each \(i \in [k]\) we have \(i \leq x_i \leq 2k-(k-i)\) and consequently \(0 \leq x_i - i \leq k\). Likewise, for each \(i \in [k-1]\) we have \(x_{i+1} \geq x_i +1\), so \(x_{i+1}-(i+1) \geq x_i - i\), and we conclude that \(\phi(X) \in L(k,k)\).
    Essentially the same argument shows that \(\phi\) has a well-defined inverse \(\phi^{-1}\left((x_i)_{i \in [k}\right) \defeq \{x_i + i\}_{i \in [k]} \), and so \(\phi\) is bijective. 
    Finally, to see that \(\phi\) is an order isomorphism, observe that \(\lc{X}{Y}\) if and only if for every \(i \in [k]\) we have \(x_i \leq y_i\), whilst \(\phi \left(\{x_i\}_{i \in [k]}\right) \preceq \phi \left(\{y_i\}_{i \in [k]}\right) \) if and only if for every \(i \in [k]\) we have \(x_i - i \leq y_i - i\), which is an identical condition.
\end{proof}

Our interest in \(L(k,k)\) is due to the following theorem in the form presented by \citet{stanley-ac}. Along with Lemma~\ref{lem:poset-iso}, gives the desired structural properties of \(\poset\).

\begin{theorem}[{\cite[Corollary~6.10]{stanley-ac}}]\label{thm:poset-props}
    For all \(m,n \in \mathbb{N}\), the poset \(L(m,n)\) is rank-symmetric, rank-unimodal and Sperner.
\end{theorem}

Indeed, the rank-symmetry of \(L(m,n)\) is easy to see by arranging the partitions in complementary pairs by defining \((x_i)_{i \in [m]}^c = (n-x_{m-i+1})_{i \in [m]} \). For example, in \(L(2,3)\) the partition 00 would pair with 33, 01 with 23, 11 with 22 and so on. The rank-unimodality of \(L(m,n)\) was first given by \citet{sylvester}. A purely combinatorial proof was given by \citet{o'hara} more than 100 years later, by presenting an injection between rank layers \(i\) and \(i + 1\) in \(L(m,n)\) (for \(i < \frac{1}{2}mn\)). The Spernicity of \(L(m,n)\) is due to \citet{stanley-sperner_property}. It would be desirable to find an injection \(f : \rho_i \to \rho_{i+1}\) (i.e. between successive rank layers) for each \(i < \frac{1}{2}{mn}\)  with the additional property that \(x \preceq f(x)\) for all \(x \in \rho_i\). This would give a symmetric chain decomposition of \(L(m,n)\), from which rank-unimodality and Spernicity would then follow immediately (see \citet{stanley-ac} for further details and the definition of symmetric chain decomposition). However, the existence of such an injection remains a significant open problem.

\begin{corollary}\label{cor:boundary}
    For even \(k\), every antichain in \(\poset\) has size at most \(2|\Pk|\).
\end{corollary}

\begin{proof}
Observe that \(L(k, k)\) has \(k^2+1\) rank layers \(\rho_0, \rho_1, \dots, \rho_{k^2}\). The rank-symmetry and rank-unimodality of \(L(k, k)\) together imply that \(p_{k^2/2} \geq p_j\) for each \(0 \leq j \leq k^2\). The fact that \(L(k, k)\) is Sperner then implies that every antichain in \(L(k, k)\) has size at most \(p_{k^2/2}\), and since \(L(k, k)\) and \(\poset\) are order isomorphic the same is true of every antichain in \(\poset\).
    Now recall the bijection \(\phi\) from the proof of Lemma~\ref{lem:poset-iso}. For every \(X = (x_i)_{i \in [k]} \in \rho_{k^2/2}\) we have \( \total{\phi^{-1}\left(X\right)} = \sum_{i \in [k]}(x_i + i) = \rho\left((x_i)_{i \in [k]}\right) + \sum_{i \in [k]} i = \frac{k^2}{2} + \frac{k(k+1)}{2} = \frac{k}{2}(2k+1)\), so \(\phi^{-1}(X) \in \{F, F^c\}\) for some \(\{F, F^c \} \in \Pk\). It follows that that \(|\rho_{k^2/2}| \leq |\bigcup \Pk| = 2|\Pk|\), and so every every antichain in \(\poset\) has size at most~\(2|\Pk|\). \qedhere
\end{proof}

In 1982, \citet{prodinger-partitions} established the asymptotic growth of \(|\Pk|\).

\begin{theorem}[\cite{prodinger-partitions}] \label{thm:asympt}
For even \(k\) we have  \(|\Pk| \sim \frac{2^{2k}\sqrt{3}}{2\pi k^2}\).
\end{theorem}

With this, we are ready to prove Theorem~\ref{thm:main_double-exp}.

\begin{proof}[Proof of Theorem~\ref{thm:main_double-exp}]
Having proved Theorem~\ref{thm:main_allk}, it remains only to establish the asymptotic bounds on the size of \(|\M{k}|\). Theorem~\ref{thm:asympt} tells us that \(|\Pk| \sim \frac{2^{2k}\sqrt{3}}{2\pi k^2}\) for even \(k\), and a standard fact on the binomial distribution is that \(\binom{2k}{k} \sim \frac{2^{2k}}{\sqrt{\pi k}}\). Combining these asymptotic estimates with the fact that \(\frac{4}{9} < \frac{\sqrt{3}}{2\sqrt{\pi}} < \frac{1}{2}\), we conclude that for sufficiently large even \(k\) we have 
\begin{equation}\label{eqpk}
    \frac{4}{9}k^{-3/2}\binom{2k}{k} < |\Pk| < \frac{1}{2} k^{-3/2}\binom{2k}{k}.
\end{equation}
By Lemma~\ref{lem:lower_bound} it follows that \(|\M{k}| \geq 2^{|\Pk|} > 2^{\frac{4}{9}k^{-3/2}\binom{2k}{k}}\) for sufficiently large even \(k\). Consequently, for sufficiently large odd \(k\) we have \(|\M{k}| \geq |\M{k-1}| > 2^{\frac{4}{9}(k-1)^{-3/2}\binom{2(k-1)}{k-1}} > 2^{\frac{1}{9}k^{-3/2}\binom{2k}{k}}\) by Lemma~\ref{lem:noofIFsGrow}, giving the desired lower bound.

We now turn to the upper bound. Let \(k\) be even and sufficiently large that~\eqref{eqpk} holds, write \(m\defeq \binom{2k}{k}\), and let \(\F \subseteq \binom{[2k]}{k}\) be an MLCIF. The family \(\B(\F)\) of boundary sets of~\(\F\) is an antichain in \(\poset\), so by Corollary~\ref{cor:boundary} we have \(|\B(\F)| \leq 2|\Pk|\). Since \(\B(\F)\) uniquely determines \(\F\), it follows that 
    \begin{align*}
    |\M{k}| & \leq \sum_{i=1}^{2|\Pk|}\binom{m}{i} \leq 2|\Pk|\binom{m}{2|\Pk|} < 
    2^{2k}\binom{m}{k^{-3/2}m}
    \leq 2^{2k} \left(\frac{me}{k^{-3/2}m} \right)^{k^{-3/2}m} \\ 
    & < 2^{2k}\left( e k^{3/2} \right)^{k^{-3/2}m} = 2^{2k+\log_2(ek^{3/2})k^{-3/2} m} < 2^{\frac{7}{4}\log_2(k)k^{-3/2}\binom{2k}{k}},
    \end{align*}
    where the first strict inequality follows from~\eqref{eqpk}.
    Consequently, by Lemma~\ref{lem:noofIFsGrow}, it follows that for sufficiently large odd~\(k\) we have 
    \[
    |\M{k}| \leq |\M{k+1}| \leq 2^{\frac{7}{4}\log_2(k+1)(k+1)^{-3/2}\binom{2(k+1)}{k+1}} < 2^{7\log_2(k)k^{-3/2}\binom{2k}{k}}. \qedhere
    \]
\end{proof}

\subsection*{Potential stronger bounds via the hypergraph container method}
Our proof of Theorem~\ref{thm:main_double-exp} bounded the number of possibilities for the antichain \(\B(\F)\) in a crude way: if \(q\) is the size of the largest antichain in \(L(k, k)\), then the number of antichains in \(L(k, k)\) is at most the number of subsets of \(L(k, k)\) with size at most \(q\). The celebrated hypergraph container method (developed by Balogh, Morris and Samotij~\cite{Balogh_Morris_Samotij} and \citet{Saxton_Thomason}) provides a more nuanced possible approach for bounding the number of antichains. Specifically, following this approach we would show that each antichain in \(L(k, k)\) is contained in one of a relatively small number of sets (the eponymous `containers'), where each container set is close to being an antichain in the sense that it contains few comparable pairs. If we could deduce from this that the size of each container set is at most a constant factor larger than \(q\), then we would obtain the desired improved bound on the number of antichains \(L(k, k)\). However, to make this deduction we require the following conjectured supersaturation result (recall for this that we showed in the proof of Theorem~\ref{thm:main_double-exp} that the size of the largest antichain in \(L(k, k)\) is around \(k^{-3/2}\binom{2k}{k}\)).

\begin{conjecture}\label{conj:supersaturation}
    There exist \(c_1, c_2 > 0\) such that the following holds for all sufficiently large \(k\). Let \(G_k\) be the graph on \(L(k,k)\) in which \(\{\lambda,\mu\} \in E(G_k)\) if \(\lambda \preceq \mu \) or \(\lambda \succeq \mu\). For every \(U \subseteq L(k,k)\) with \(|U| \geq c_1 k^{-3/2}\binom{2k}{k}\), we have \(e(G_k[U]) \geq c_2 k^2 |U|\).
\end{conjecture}

The analogous result for the Boolean lattice was obtained by \citet{Kleitman}, and Noel, Scott and Sudakov~\cite{Noel_Scott_Sudakov} provided a general method for obtaining supersaturation results in various other posets. However, known structural results for the Young lattice \(L(m,n)\) remain fairly limited. For example, it is still unknown if \(L(m,n)\) has a symmetric chain decomposition as was conjectured to be the case by \citet{stanley-sperner_property} in 1980. As a result, we have so far been unable to prove Conjecture~\ref{conj:supersaturation}.

If Conjecture~\ref{conj:supersaturation} is true, then we could follow the hypergraph container argument outlined above using the same approach as was applied in the Boolean lattice by Balogh, Mycroft and Treglown~\cite{Balogh_Mycroft_Treglown}. This would yield an upper bound of \(2^{c k^{-3/2}\binom{2k}{k}}\) on the number of antichains in \(L(k, k)\) (for some constant \(c\)); by using this stronger bound in the proof of Theorem~\ref{thm:main_double-exp} we would then conclude that \(|\M{k}| = 2^{\Theta \left(k^{-3/2}\binom{2k}{k}\right)}\).
%
%
\begin{singlespace}
	\bibliographystyle{plainnat}
	\bibliography{ref}
\end{singlespace}

\end{document}